\documentclass{amsart}
\usepackage{amsmath,amssymb}
\newcommand{\Sym}{\mathrm{Sym}}
\newcommand{\Alt}{\mathrm{Alt}}
\newcommand{\Cos}{\mathop{\mathrm{Cos}}}
\renewcommand{\wr}{\mathop{\mathrm{wr}}}
\newtheorem{theorem}{Theorem}
\newtheorem{lemma}[theorem]{Lemma}

\newtheorem{corollary}[theorem]{Corollary}

\newtheorem{hyp}[theorem]{Hypothesis}
\newcommand{\Aut}{\mathop{\mathrm{Aut}}}
\begin{document}
\title[Local conditions and Sylow subgroups]{Two local conditions on the vertex stabiliser of arc-transitive graphs and their effect on the Sylow subgroups}

\author[P. Spiga]{Pablo Spiga}
\address{Pablo Spiga,  School of Mathematics and Statistics,\newline
The University of Western Australia,
 Crawley, WA 6009, Australia} \email{spiga@maths.uwa.edu.au}


\begin{abstract} In this paper we study $G$-arc-transitive graphs $\Delta$ where the permutation group $G_x^{\Delta(x)}$ induced by the stabiliser $G_x$ of the vertex $x$ on the neighbourhood $\Delta(x)$ satisfies the two conditions given in the introduction. We show that for such a  $G$-arc-transitive graph $\Delta$, if $(x,y)$ is an arc of $\Delta$, then the subgroup $G_{x,y}^{[1]}$ of $G$ fixing pointwise $\Delta(x)$ and $\Delta(y)$ is a $p$-group for some prime $p$. Next we prove that every $G$-locally primitive (respectively quasiprimitive, semiprimitive) graph satisfies our two local hypotheses. Thus this provides a new Thompson-Wielandt-like theorem for a very large class of arc-transitive graphs.

Furthermore, we give various families of $G$-arc-transitive graphs  where our two local conditions do not apply and where $G_{x,y}^{[1]}$ has arbitrarily large composition factors. 
\end{abstract}

\thanks{Address correspondence to P. Spiga: spiga@maths.uwa.edu.au\\ 
The author is supported by UWA as part of the  Australian Research Council Federation Fellowship Project  FF0776186.}

\maketitle
\section{Introduction}\label{sec:1}

A graph $\Delta$ is said to be $G$-\emph{vertex-transitive} if $G$ is a subgroup of $\Aut(\Delta)$ acting transitively on the vertex set $V\Delta$ of $\Delta$. Similarly,  $\Delta$ is said to be $G$-\emph{arc-transitive} if $G$ acts transitively on the arcs of $\Delta$, that is, on the ordered pairs of adjacent vertices of $\Delta$. We say that a $G$-vertex-transitive graph $\Delta$ is $G$-\emph{locally primitive} if the stabiliser $G_x$ of the vertex $x$ induces a primitive
permutation group on the set $\Delta(x)$ of vertices adjacent to $x$. In 1978 Richard Weiss~\cite{Weiss} conjectured that for a finite connected $G$-vertex-transitive, $G$-locally
primitive graph $\Delta$, the size of $G_x$ is bounded above by some function depending only on the valency of $\Delta$. This conjecture is very similar to the 1967 conjecture
of Charles Sims~\cite{Sims}, that  for a $G$-vertex-primitive graph or digraph $\Delta$, the size of the stabiliser of a vertex is bounded above by some function depending only on the valency of $\Delta$. Despite the fact that the Sims Conjecture has
been proved true in~\cite{CPSS}, the truth of the Weiss Conjecture is still unsettled and only partial results are known.

In $1998$ Cheryl Praeger~\cite{Cheryl} made a conjecture stronger than the Weiss conjecture.

\smallskip

\noindent\textbf{Praeger Conjecture. }{\em There exists a function $f : \mathbb{N} \to \mathbb{N}$ such that, if $\Delta$  is a connected $G$-vertex-transitive, $G$-locally quasiprimitive graph of valency $d$ and $x\in V\Delta$, then $|G_x|\leq f(d)$.}
\medskip

From the Weiss Conjecture to  the Praeger Conjecture the local action assumption is weakened from primitive to quasiprimitive. (We recall that a permutation group $L$ is \emph{quasiprimitive} if every non-identity normal subgroup of $L$ is transitive.) Furthermore, recently in~\cite{PSV} the following generalisation of both the Weiss and the Praeger Conjecture was made. (A finite permutation group $L$ is \emph{semiprimitive} if every normal subgroup of $L$ is either transitive or semiregular. We refer the reader to~\cite{BM,KS} for the original motivation for introducing and studying the semiprimitive groups, for some interesting examples and for some 
group theoretic results.)

\smallskip

\noindent\textbf{Conjecture~A.~\cite{PSV} }{\em There exists a function $f : \mathbb{N} \to \mathbb{N}$ such that, if $\Delta$  is a connected $G$-vertex-transitive, $G$-locally semiprimitive graph of valency $d$ and $x\in V\Delta$, then $|G_x|\leq f(d)$.}

\medskip

The proof of the Sims Conjecture uses the
Classification of Finite Simple Groups and  at a crucial stage an
important theorem of Thompson~\cite[Theorem]{Thompson} concerning the 
structure of the stabiliser of a point in a primitive group. A first
generalisation of the theorem of Thompson was obtained by
Wielandt~\cite[Theorem~$6.6$]{Wielandt}. In graph theoretic terminology, the Thompson-Wielandt theorem states that for a $G$-vertex-primitive digraph $\Delta$, if $(x,y)$ is an arc of $\Delta$, then the subgroup of $G$ fixing pointwise $\Delta(x)$ and $\Delta(y)$ is a $p$-group for some prime $p$. This result
 inspired much research and various generalisations
of the so called \emph{Thompson-Wielandt} theorems have been obtained by
many authors~\cite{G,K,VanBon,W}.

In this paper we obtain a new
Thompson-Wielandt theorem for a very large class of graphs (we refer to Section~\ref{sec:not} for undefined terminology). 

\begin{hyp}\label{hyp1}{\rm
Let $\Delta$ be a connected graph and $G$ a group of automorphisms of $\Delta$. Suppose that for each vertex $x$ of $\Delta$, the group $G_x^{\Delta(x)}$ is transitive on $\Delta(x)$ and satisfies the following two conditions:
\begin{description}
\item[$(i)$] $(C_{G_x}(G_x^{[1]}))^{\Delta(x)}$ is transitive or semiregular on $\Delta(x)$;
\item[$(ii)$] for each prime $p$, the group $\langle O_p(G_{x,y})\mid y\in \Delta(x)\rangle^{\Delta(x)}$ is transitive or semiregular on $\Delta(x)$.
\end{description}
Since for each $x\in V\Delta$ the group $G_x^{\Delta(x)}$ is transitive, we obtain that $G$ acts transitively on the edges of $\Delta$, that is, $\Delta$ is $G$-edge-transitive. We stress here that we do not require $\Delta$ to be $G$-arc-transitive.}
\end{hyp}

The conditions in Hypothesis~\ref{hyp1} are satisfied by many important classes of arc-transitive graphs. For example, as $C_{G_x}(G_x^{[1]})$ and $\langle O_p(G_{x,y})\mid y\in \Delta(x)\rangle$ are normal subgroups of $G_x$, we obtain that if $\Delta$ is $G$-locally primitive (or quasiprimitive, or semiprimitive), then $\Delta$ and $G$ satisfy Hypothesis~\ref{hyp1}. 

One of the main results of this paper is the following theorem.

\begin{theorem}\label{thm:main}
Let $\Delta$ and $G$ be as in Hypothesis~\ref{hyp1} and let $(x,y)$ 
be an arc of $\Delta$. Then either $G_{x,y}^{[1]}$ is a $p$-group, or
(interchanging $x$ and $y$ if necessary) $G_{x,y}^{[1]}=G_{x}^{[2]}$ and the
group $G_y^{[2]}=G_y^{[3]}$ is a $p$-group. Furthermore either $F^*(G_{x,y})$
is a $p$-group, or (interchanging $x$ and $y$ if necessary)
$G_{x,y}^{[1]}=G_x^{[2]}$ and $G_y^{[2]}=1$. 
\end{theorem}

The following is an immediate corollary of Theorem~\ref{thm:main}.
\begin{corollary}\label{cor:1}Let $\Delta$ be a connected $G$-vertex-transitive, $G$-locally semiprimitive graph and $(x,y)$ an arc of $\Delta$. Then $G_{x,y}^{[1]}$ is a $p$-group for some prime $p$. Furthermore either $F^*(G_{x,y})$ is a $p$-group or $G_{x,y}^{[1]}=1$. 
\end{corollary}

Theorem~\ref{thm:main} and Corollary~\ref{cor:1} are already known for $G$-vertex-transitive, $G$-locally primitive (respectively quasiprimitive) graphs, see for example~\cite[Theorem]{Thompson},~\cite[Theorem~$1.1$]{VanBon} and~\cite[Theorem~$1$]{W}. Therefore Corollary~\ref{cor:1} can be viewed as a natural generalisation of well-established theorems on locally primitive and quasiprimitive graphs to locally semiprimitive graphs. 

Furthermore, in the light of Corollary~\ref{cor:1} the Praeger and the Weiss conjecture and Conjecture~A ask whether the size of the $p$-group $G_{x,y}^{[1]}$ is bounded above by some function of the valency of $\Delta$.   

In general, for a connected $G$-arc-transitive graph and for an arc $(x,y)$, the group $G_{x,y}^{[1]}$ is not necessarily a  $p$-group. In this paper, we provide new remarkable examples of $G$-arc-transitive graphs $\Delta$ where $G_{x,y}^{[1]}$ has arbitrarily many and arbitrarily large composition factors  and furthermore with $G$ being an almost simple group. Indeed, we prove the following theorem.

\begin{theorem}\label{thm:2}Let $rs$  be a composite integer (where $r,s>1$), $R$ and $S$ transitive permutation groups of degree $r$ and  $s$ (respectively), and $m\geq 3$ an odd integer. There exists a connected $G$-arc-transitive graph $\Delta$ of valency $rs$ with $G$ almost simple such that, for an arc $(x,y)$ of $\Delta$, $G_x^{\Delta(x)}$ is permutation isomorphic to $R\wr S$ (in its imprimitive action of degree $rs$) and $G_{x,y}^{[1]}\cong R^{s(m-2)+1}$. 
\end{theorem}

Theorem~\ref{thm:2} shows that, in the very restricted class of graphs admitting an \emph{arc-transitive almost simple} group of automorphisms $G$, some assumptions on the local action of $G$ are needed for guaranteeing that  $G_{x,y}^{[1]}$ is a $p$-group . A similar comment applies to the Praeger and the Weiss Conjecture and to Conjecture~A. Namely, since $m$ in Theorem~\ref{thm:2} does not depend on the valency $rs$, the size of the group $G_{x,y}^{[1]}$ is not bounded by a function of $rs$. 

\subsection{Structure of the paper}
Section~\ref{sec:2} consists of seven technical lemmas whose proof is
obtained by  adapting, to the more general context of Hypothesis~\ref{hyp1},  Lemmas~$2.1$--$2.7$ in~\cite{VanBon}. Since the paper is short and
elementary, we give a full argument although the proofs of our lemmas are essentially as in~\cite{VanBon}. 
The proofs of Theorem~\ref{thm:main} and
Corollary~\ref{cor:1} are in 
Section~\ref{sec:3}. Finally, Theorem~\ref{thm:2} follows at once from Theorem~\ref{thm:3} in Section~\ref{sec:4}.

\subsection{Notation}\label{sec:not}
In this subsection we fix some notation that we use in the rest of the
paper. Let $\Delta$ be a finite connected graph and 
$G$ a group of automorphisms of $\Delta$. Given a vertex $z$
of $\Delta$, we denote by $\Delta(z)$ the set of vertices adjacent to
$z$, by $G_z$ the stabiliser in $G$ of $z$ and (for $i\geq 0$) by
$G_z^{[i]}$ the group   
$$\{g\in G_z\mid y^g=y\textrm{ for all }y\in \Delta \textrm{
with }y \textrm{ at distance at most }i \textrm{ from }z\}.$$
Given an arc $(x,y)$ of $\Delta$ we define
$G_{x,y}^{[1]}=G_x^{[1]}\cap G_y^{[1]}$, that is, the subgroup of $G$ fixing pointwise $\Delta(x)$ and $\Delta(y)$.

We denote
by $F^*(G)$ the generalised Fitting subgroup of $G$, by $F(G)$ the
Fitting subgroup of $G$ and by $E(G)$ the subgroup  generated by
the components of $G$ (that is, the subnormal quasisimple subgroups of
$G$). We refer 
to~\cite[Chapter~$6$ \S 6]{Suzuki}  
for definitions and for basic results on $F^*(G),F(G)$ and $E(G)$.  

Finally, a permutation group $L$ is semiregular if the identity is the only element of $L$ fixing some point.

\section{Lemmata}\label{sec:2}
Recall that Notation~\ref{sec:not} is assumed throughout.
\begin{lemma}\label{lemma:0}
Let $(x,y)$ be an arc of $\Delta$. Suppose that $N$ is a subgroup
of $G_x$ acting transitively on $\Delta(x)$ and 
 $M$ is a subgroup of $G_y$ acting transitively on
$\Delta(y)$. If $H\leq G_{x,y}$ and $H\vartriangleleft\langle
N,M\rangle$, then $H=1$. 
\end{lemma}

\begin{proof}
Set $L=\langle N,M\rangle$. If $x$ and $y$ are in the same $L$-orbit,
then $L$ acts transitively on the vertices of
$\Delta$. In particular, $L_{x}$ is a core-free subgroup of $L$ and
$H=1$. If $x$ and $y$ are in distinct $L$-orbits, 
then $L$ has exactly two orbits $x^L$ and $y^L$ on the vertices of
$\Delta$. In particular, $L_{x,y}$ is a core-free subgroup of $L$ and  $H=1$. 
\end{proof}

\begin{lemma}\label{lemma:2.2}
Let $(x,y)$ be an arc of $\Delta$. If $E(G_x^{[1]})\neq E(G_{x,y})$, then
$G_{x,y}^{[1]}=G_x^{[2]}$ and $G_y^{[2]}=1$. Moreover, if
$E(G_x^{[1]})=E(G_{x,y})=E(G_y^{[1]})$, then
$E(G_{x,y})=1$. 
\end{lemma}
\begin{proof}
Suppose that $E(G_{x,y})\neq E(G_{x}^{[1]})$. Since $G_x^{[1]}$ is a
normal subgroup of $G_{x,y}$, we have $E(G_x^{[1]})\leq
E(G_{x,y})$. As $E(G_{x,y})\neq E(G_{x}^{[1]})$, there exists  a
component $K$ of $G_{x,y}$ with $K\not\leq
G_x^{[1]}$. Then, by~\cite[Lemma~$6.9$~$(iv)$]{Suzuki}, we have 
$[K,G_x^{[1]}]=1$. Since $K$ is not contained in $G_x^{[1]}$, the
group $K$ acts non-trivially
on $\Delta(x)$. As $C_{G_x}(G_x^{[1]})$ contains $K$, the group
$C_{G_x}(G_x^{[1]})$ acts non-trivially on $\Delta(x)$. 

Set $N=C_{G_x}(G_x^{[1]})$. By Hypothesis~\ref{hyp1}~$(i)$, it follows that $N^{\Delta(x)}$ is either
transitive or semiregular on $\Delta(x)$. Since $K\leq N$, $K\leq
G_{x,y}$ and $K\not\leq G_x^{[1]}$, the group
$N$ is not semiregular on $\Delta(x)$. Thence $N$ acts
transitively on $\Delta(x)$. Since $y$ is adjacent to $x$, the group
$G_y^{[2]}$ is  a subgroups of $G_x^{[1]}$ and so $N$ centralises
$G_y^{[2]}$. In particular, $G_y^{[2]}\leq G_{x,y}$ and
$G_y^{[2]}\vartriangleleft\langle N,G_y\rangle$. By 
Lemma~\ref{lemma:0}, we get 
$G_y^{[2]}=1$. Since $N$ is transitive on $\Delta(x)$ and centralises
$G_{x,y}^{[1]}$, we have 
$G_{x,y}^{[1]}=G_{x,y'}^{[1]}$  for every $y,y'$ in $\Delta(x)$. This
yields  $G_{x,y}^{[1]}=G_x^{[2]}$ and the first part of the lemma
is proved. 

If $E(G_x^{[1]})=E(G_{x,y})=E(G_y^{[1]})$, then $E(G_{x,y}^{[1]})\leq G_{x,y}$ and
$E(G_{x,y}^{[1]})\vartriangleleft\langle G_x,G_y\rangle$,
and hence $E(G_{x,y}^{[1]})=1$ by Lemma~\ref{lemma:0}. 
\end{proof}

Let $(x,y)$ be an arc of $\Delta$, and let $p$ be a prime. Define
\begin{eqnarray*}
S_{x,y}&=&O_p(G_{x,y}),\\
E_x&=&\langle S_{x,y}^g\mid g \in G_x \rangle=\langle S_{x,z}\mid z\in
\Delta(x)\rangle,\\ 
E_y&=&\langle
S_{x,y}^g\mid g \in G_y\rangle=\langle S_{z,y}\mid z\in
\Delta(y)\rangle.
\end{eqnarray*} 

\begin{lemma}\label{lemma:2.3}Let $p$ be a prime. If $x$ is a vertex
  of $\Delta$, then $[G_x^{[1]},E_x]\leq O_p(G_x^{[1]})$.
\end{lemma}

\begin{proof}Let $z$ be in $\Delta(x)$. The groups $G_x^{[1]}$ and
  $S_{x,z}$ are  normal subgroups of $G_{x,z}$. Hence $[G_x^{[1]},S_{x,z}]\leq
  G_x^{[1]}\cap S_{x,z}$. Since $G_x^{[1]}\cap S_{x,z}$ is a normal
  $p$-subgroup of $G_x^{[1]}$, we get $G_x^{[1]}\cap S_{x,z}\leq
  O_p(G_x^{[1]})$. Finally, as $G_x$ normalises $G_x^{[1]}$, we have
\begin{eqnarray*}
[G_x^{[1]},E_x]&=&[G_{x}^{[1]},\langle S_{x,y}^g\mid g \in
  G_x\rangle]\leq \langle [G_x^{[1]},S_{x,y}]^g\mid g \in
G_x\rangle\leq O_p(G_x^{[1]}).
\end{eqnarray*}
\end{proof}

\begin{lemma}\label{lemma:mio}
If $E_x\not\leq G_x^{[1]}$, then $E_x$ is transitive on $\Delta(x)$
and, for every $r$-subgroup $R$ of $G_x^{[1]}$
for a prime $r\neq p$,
the group $N_{G_x}(R)$ is transitive on $\Delta(x)$.
\end{lemma}

\begin{proof}
Let $y$ be in $\Delta(x)$.  By
Hypothesis~\ref{hyp1}~$(ii)$, 
the group $E_x$ is either transitive or 
semiregular on $\Delta(x)$. Assume $E_x$ semiregular on
$\Delta(x)$. Since $S_{x,y}$ is
contained in $E_x$ and 
fixes the neighbour $y$ of $x$, we get that $S_{x,y}$ fixes pointwise
$\Delta(x)$ and $S_{x,y}\leq G_x^{[1]}$. Since $G_x^{[1]}$
is normal in $G_x$ and $E_x=\langle S_{x,y}^g\mid g\in G_x\rangle$,
we obtain $E_x\leq G_x^{[1]}$, a
contradiction. This yields that $E_x$ is transitive on
$\Delta(x)$.

Since $G_x^{[1]}$ is a normal subgroup of $G_{x,y}$, we get that
$O_p(G_x^{[1]})$ is contained in $O_p(G_{x,y})=S_{x,y}\leq
E_x$.  Let $R$ be an
$r$-subgroup of $G_x^{[1]}$ with $r\neq p$. From
Lemma~\ref{lemma:2.3}, we get $[R,E_x]\leq [G_x^{[1]},E_x]\leq
O_p(G_x^{[1]})$. So, $E_x$ normalises
$RO_p(G_x^{[1]})$. Since $R$ is a Sylow $r$-subgroup of
$RO_p(G_x^{[1]})$ and $O_p(G_x^{[1]})\leq E_x$, by the Frattini
argument, we obtain $E_x=N_{E_x}(R)O_p(G_x^{[1]})$ and $N_{E_x}(R)$
acts transitively on 
$\Delta(x)$.  
\end{proof}

\begin{lemma}\label{lemma:2.5}
Let $p$ be a prime and $(x,y)$ an arc of $\Delta$.
\begin{description}
\item[$(i)$]If $E_x\leq G_x^{[1]}$ and $E_y\leq G_y^{[1]}$, then  $O_p(G_{x,y})=1$.
\item[$(ii)$]If $E_x\not\leq G_x^{[1]}$ and $E_y\not\leq G_y^{[1]}$, then
  $G_{x,y}^{[1]}$ is a $p$-group.
\end{description}
\end{lemma}
\begin{proof}
We first prove $(i)$. Assume $E_x\leq G_x^{[1]}$ and $E_y\leq
G_y^{[1]}$. We have
$O_p(G_{x,y})=S_{x,y}\leq E_x\leq 
G_{x}^{[1]}$. Thence $O_p(G_{x,y})$ is a normal $p$-subgroup of
$G_x^{[1]}$ and so $O_p(G_{x,y})\leq O_p(G_x^{[1]})$. Also, since
$G_x^{[1]}$ is normal in $G_{x,y}$, we obtain 
$O_p(G_x^{[1]})\leq O_p(G_{x,y})$. Thence 
$O_p(G_{x,y})=O_p(G_x^{[1]})$. By symmetry between $x$ and $y$,
$O_p(G_{x,y})=O_p(G_y^{[1]})$. In particular, $O_p(G_{x,y})$ is a
normal subgroup of $G_x$ and $G_y$. Therefore $O_p(G_{x,y})\leq
G_{x,y}$ and
$O_p(G_{x,y})\vartriangleleft\langle G_x,G_y\rangle$. Hence
Lemma~\ref{lemma:0} yields $O_p(G_{x,y})=1$.

Now we prove~$(ii)$. Let
$R$ be a Sylow $r$-subgroup of $G_{x,y}^{[1]}$ with $r\neq
p$. Lemma~\ref{lemma:mio} implies that $N_{G_x}(R)$ acts transitively
on $\Delta(x)$ and $N_{G_y}(R)$ acts transitively on
$\Delta(y)$. Since $R\leq G_{x,y}$ and 
$R\vartriangleleft\langle N_{G_x}(R),N_{G_y}(R)\rangle$,
Lemma~\ref{lemma:0} yields $R=1$.  This proves that
either $G_{x,y}^{[1]}=1$ or  $p$ is the only prime dividing the
order of $G_{x,y}^{[1]}$. 
\end{proof}

\begin{lemma}\label{lemma:2.6}
Let $p$ be a prime and $(x,y)$ an arc of $\Delta$. If $E_x\not\leq
G_x^{[1]}$, then $G_y^{[2]}$ is a $p$-group.
\end{lemma}
\begin{proof}
Let $R$ be a Sylow $r$-subgroup of $G_y^{[2]}$ with $r\neq
p$. Since $R$ is an $r$-subgroup of $G_x^{[1]}$, 
Lemma~\ref{lemma:mio} yields that $N_{G_x}(R)$ acts transitively on
$\Delta(x)$. The
Frattini argument yields 
$G_y=G_y^{[2]}N_{G_y}(R)$. In particular, $N_{G_y}(R)$ acts
transitively on $\Delta(y)$. Since $R\leq G_{x,y}$ and
$R\vartriangleleft\langle N_{G_x}(R),N_{G_y}(R)\rangle$,
Lemma~\ref{lemma:0} yields  $R=1$. This shows that $G_y^{[2]}$ is a $p$-group.
\end{proof}

\begin{lemma}\label{lemma:2.7}
Let $p$ be a prime and $(x,y)$ an arc of $\Delta$. If $E_x\not\leq
G_x^{[1]}$ and $E_y\leq G_y^{[1]}$, then 
$G_{x,y}^{[1]}=G_x^{[2]}$ and $G_y^{[2]}=G_y^{[3]}$.
\end{lemma}

\begin{proof}
Since $E_y\leq G_y^{[1]}$, we have $O_p(G_{x,y})=S_{x,y}\leq E_y\leq
G_y^{[1]}$. Also, as $G_x^{[1]}$ is normal in $G_{x,y}$, we get
$O_p(G_x^{[1]})\leq O_p(G_{x,y})$.  In 
particular, we obtain
$O_p(G_x^{[1]})\leq O_p(G_{x,y})\leq G_y^{[1]}$. Since 
 $O_p(G_x^{[1]})$ is normal in $G_x$, $G_x$ is transitive on $\Delta(x)$
and $O_p(G_x^{[1]})\leq G_y^{[1]}\cap G_x^{[1]}=G_{x,y}^{[1]}$, we get
$O_p(G_x^{[1]})\leq G_x^{[2]}$. From Lemma~\ref{lemma:2.3}, we obtain
$$[G_{x,y}^{[1]},E_x]\leq [G_x^{[1]},E_x]\leq 
O_p(G_x^{[1]})\leq G_x^{[2]}\leq G_{x,y}^{[1]}.$$ Hence $E_x$
normalises $G_{x,y}^{[1]}$. As $E_x\not\leq G_x^{[1]}$, by
Lemma~\ref{lemma:mio}, $E_x$ 
acts transitively on $\Delta(x)$. Thence 
$G_{x,z}^{[1]}=G_{x,y}^{[1]}$ for all $z\in \Delta(x)$. This implies
that $G_{x,y}^{[1]}=G_x^{[2]}$. Finally, let $(y,a,b)$ be a path of
length $2$ in $\Delta$. Since $(a,y)=(x,y)^g$ for some $g\in G$, we 
have $G_y^{[2]}\leq G_{a,y}^{[1]}=G_a^{[2]}\leq G_b^{[1]}$. From this it follows
that $G_y^{[2]}=G_y^{[3]}$.
\end{proof}

\section{Proofs Theorem~\ref{thm:main} and Corollary~\ref{cor:1}}\label{sec:3}

\noindent\emph{Proof of Theorem~\ref{thm:main}. }
If $G_{x,y}=1$, then there is result is clear. So, from now on
we may assume that $G_{x,y}\neq 1$. If $E(G_x^{[1]})\neq E(G_{x,y})$, then
by Lemma~\ref{lemma:2.2}  we get $G_{x,y}^{[1]}=G_x^{[2]}$ and
$G_y^{[2]}=1$. Similarly, if $E(G_y^{[1]})\neq E(G_{x,y})$, then
$G_{x,y}^{[1]}=G_y^{[2]}$ and $G_x^{[2]}=1$. In particular, from now
on we may assume that $E(G_x^{[1]})=E(G_{x,y})=E(G_y^{[1]})$. From
Lemma~\ref{lemma:2.2}, we get
$E(G_{x,y})=1$. In particular, 
$F^*(G_{x,y})=E(G_{x,y})F(G_{x,y})=F(G_{x,y})$. Thus there
exists a prime $p$ 
with $O_p(G_{x,y})\neq 1$. Consider for this prime $p$ the groups
$E_x$ and $E_y$.  If $E_x\not\leq G_x^{[1]}$ and $E_y\not\leq G_x^{[1]}$, then
by Lemma~\ref{lemma:2.5}~$(ii)$, the group $G_{x,y}^{[1]}$ is a
$p$-group. As $O_p(G_{x,y})\neq 1$, by Lemma~\ref{lemma:2.5}~$(i)$,
interchanging $x$ 
and $y$ if necessary, we may assume that $E_x\not\leq G_x^{[1]}$ and $E_y\leq
G_y^{[1]}$. Lemma~\ref{lemma:2.6} shows that $G_y^{[2]}$ is a 
$p$-group and Lemma~\ref{lemma:2.7} gives $G_{x,y}^{[1]}=G_x^{[2]}$ and $G_y^{[2]}=G_y^{[3]}$. This completes the proof of the first
part of the theorem.

As $F^*(G_{x,y})=F(G_{x,y})$, it remains to study the group $F(G_{x,y})$. If $F(G_{x,y})$ is a $p$-group for some prime $p$, then there is nothing to prove. Therefore we may assume that there exist two distinct primes $p$ and $r$ 
with $O_p(G_{x,y}),O_r(G_{x,y})\neq 1$. Consider, for these primes $p$
and $r$, the groups
$E_{p,x},E_{p,y},E_{r,x}$ and $E_{r,y}$. As $O_p(G_{x,y})\neq 1$, by
Lemma~\ref{lemma:2.5}~$(i)$, interchanging $x$ 
and $y$ if necessary, we may assume that $E_{p,x}\not\leq
G_x^{[1]}$. Similarly, as $O_r(G_{x,y})\neq 1$, by 
Lemma~\ref{lemma:2.5}~$(i)$, either
$E_{r,x}\not\leq G_x^{[1]}$ or $E_{r,y}\not\leq G_y^{[1]}$. In the
rest of the proof we study separately these two cases.

\noindent\textsc{Case $E_{r,x}\not\leq G_x^{[1]}$. }If $E_{p,y}\leq
G_{y}^{[1]}$, then, from Lemmas~\ref{lemma:2.6} and~\ref{lemma:2.7}
(applied to the prime $p$) we get that $G_{x,y}^{[1]}=G_x^{[2]}$ and
$G_y^{[2]}=G_y^{[3]}$ is a $p$-group. As $E_{r,x}\not\leq G_x^{[1]}$,
by Lemma~\ref{lemma:2.6} (applied to 
the prime $r$) we obtain that $G_y^{[2]}$ is an $r$-group. Thus
$G_y^{[2]}=1$.  

If $E_{r,y}\leq G_y^{[1]}$, then from
Lemmas~\ref{lemma:2.6} and~\ref{lemma:2.7} 
(applied to the prime $r$) we get that $G_{x,y}^{[1]}=G_x^{[2]}$ and
$G_y^{[2]}=G_y^{[3]}$ is an $r$-group. As $E_{p,x}\not\leq
G_{x}^{[1]}$, by Lemma~\ref{lemma:2.6} (applied to
the prime $p$) we obtain that $G_y^{[2]}$ is a $p$-group. Thus
$G_y^{[2]}=1$.  

Finally, if $E_{p,y}\not\leq G_y^{[1]}$ and $E_{r,y}\not\leq G_y^{[1]}$, then
Lemma~\ref{lemma:2.5}~$(ii)$ (applied to the prime $p$ and $r$)
yields $G_{x,y}^{[1]}$ is a $p$-group and an $r$-group. So
$G_{x,y}^{[1]}=1$ and $G_x^{[2]}=G_y^{[2]}=1$.

\noindent\textsc{Case $E_{r,x}\leq G_{x}^{[1]}$ and $E_{r,y}\not\leq
  G_y^{[1]}$. } From
Lemmas~\ref{lemma:2.6} and~\ref{lemma:2.7} 
(applied to the prime $r$) we get that $G_{x,y}^{[1]}=G_y^{[2]}$ and
$G_x^{[2]}=G_x^{[3]}$ is an $r$-group. If $E_{p,y}\not\leq G_y^{[1]}$,
then by Lemma~\ref{lemma:2.5}~$(ii)$ (applied to the prime $p$) we get
$G_{x,y}^{[1]}$ is a $p$-group. Thence $G_{x}^{[2]}$ is a $p$-group
and so $G_x^{[2]}=1$. Finally, if $E_{p,y}\leq G_y^{[1]}$, then from
Lemmas~\ref{lemma:2.6} and~\ref{lemma:2.7} 
(applied to the prime $p$) we get that $G_{x,y}^{[1]}=G_x^{[2]}$ and
$G_y^{[2]}=G_y^{[3]}$ is a $p$-group. Thus $G_{x,y}^{[1]}$ is a
$p$-group and an $r$-group. So $G_{x,y}^{[1]}=1$ and
$G_x^{[2]}=G_y^{[2]}=1$.\qed\\ 

\noindent\emph{Proof of Corollary~\ref{cor:1}. }As $\Delta$ is $G$-locally semiprimitive, $\Delta$ and $G$ satisfy Hypothesis~\ref{hyp1}. Since $G$ acts
arc-transitively on $\Delta$, we have that $G_x^{[2]}$ is conjugate to
$G_y^{[2]}$. From Theorem~\ref{thm:main}, we have that $G_{x,y}^{[1]}$ 
is a $p$-group and that either
$F^*(G_{x,y})$ is a $p$-group or $G_{x,y}^{[1]}=G_x^{[2]}=1$.\qed\\

\section{The construction}\label{sec:4}
In this section we prove Theorem~\ref{thm:2} by exhibiting a very interesting family of arc-transitive graphs. Before embarking in this elaborate construction (which generalises the arc-transitive graphs given in~\cite[Section~$5$]{PSV2}) we recall the definition of \emph{coset graph} (see~\cite{Sabidussi}). Let $G$ be a group, $H$ a subgroup of $G$ and $a\in G$. The coset digraph $\Delta=\Cos(G,H,a)$ is the digraph with vertex set the right cosets of $H$ in $G$ and with arcs the ordered pairs $(Hx,Hy)$ such that $Hyx^{-1}H\subseteq HaH$ (where $HaH=\{hak\mid h,k\in H\}$). 
It is immediate to check that  $\Delta$ is undirected if and only if $a^{-1}\in HaH$, and $\Delta$ is connected if and only if $G=\langle H,a\rangle$.
Also the action of $G$ by right multiplication on $G/H$ induces an arc-transitive automorphism group of $\Delta$. Finally, $\Delta$ has valency $|H:(H\cap H)^a|$ and the action of the vertex-stabiliser $G_H=H$ on the neighbourhood $\Delta(H)$ of the vertex $H\in V\Delta$ is permutation isomorphic to the action of $H$ on the right cosets of $H\cap H^a$.

Now we are ready to introduce the graphs satisfying the statement of Theorem~\ref{thm:2}.  Let $R$ be a transitive permutation group of degree $r>1$ acting on the set $\{0,\ldots,r-1\}$, $S$ a transitive permutation group of degree $s>1$ acting on the set $\{0,\ldots,s-1\}$, $m\geq 3$ an odd integer and $\Omega=\{0,\ldots,mrs+r-2\}$. Let $\{X_0,\ldots,X_m\}$ be the partition of $\Omega$ defined by
\begin{eqnarray*}
X_j&=&\{jrs,1+jrs,\ldots, (rs-1)+jrs\}\qquad\textrm{for }0\leq j\leq m-1,\\
X_m&=&\{mrs,1+mrs,\ldots,(r-2)+mrs\}.
\end{eqnarray*}
In particular, for $j\in \{0,\ldots,m-1\}$, we have $|X_j|=rs$, and $|X_m|=r-1$. Now, for each $j\in \{0,\ldots,m-1\}$, let $\{Y_{0,j},\ldots,Y_{s-1,j}\}$ be the partition of  $X_j$ defined by
\begin{eqnarray*}
Y_{i,j}&=&\{ir+jrs,1+ir+jrs,\ldots, (r-1)+ir+jrs\}\qquad\textrm{for }0\leq i\leq s-1.
\end{eqnarray*}
In particular, for $i\in \{0,\ldots,s-1\}$ and $j\in \{0,\ldots,m-1\}$, we have $|Y_{i,j}|=r$. Also $\{Y_{i,j},X_m\mid 0\leq i\leq s-1,0\leq j\leq m-1\}$ is a partition of $\Omega$.

For each $i\in \{0,\ldots,s-1\},j\in \{0,\ldots,m-1\}$ and $\sigma\in R$, define the following permutation of $\Sym(\Omega)$
\[
x_{\sigma,i,j}:\left\{
\begin{array}{lllcl}
z+ir+jrs&\mapsto& z^{\sigma}+ir+jrs&&\textrm{for }0\leq z\leq r-1,\\
\omega&\mapsto&\omega&&\textrm{for }\omega\in \Omega\setminus Y_{i,j}.   
\end{array}
\right.
\]
The permutation $x_{\sigma,i,j}$ has support contained in $Y_{i,j}$ and its action on $Y_{i,j}$ is equivalent to the action of $\sigma$ on $\{0,\ldots,r-1\}$.
Write $$R(Y_{i,j})=\langle x_{\sigma,i,j}\mid \sigma\in R\rangle.$$
Clearly,  $R(Y_{i,j})$ fixes pointwise $\Omega\setminus Y_{i,j}$ and the action of $R(Y_{i,j})$ on $Y_{i,j}$ is permutation isomorphic to the action of $R$ on $\{0,\ldots,r-1\}$. 

Similarly, for each $\sigma\in R_{r-1}$ define the following permutation of $\Sym(\Omega)$
\[
x_{\sigma}:\left\{
\begin{array}{lllll}
\omega&\mapsto&\omega&&\textrm{for }\omega\in \Omega\setminus X_m,\\
z+mrs&\mapsto &z^{\sigma}+mrs&&\textrm{for }0\leq z\leq r-2.\\   
\end{array}
\right.
\]
Write $$R(X_m)=\langle x_{\sigma}\mid \sigma\in R_{r-1}\rangle.$$  Clearly,  $R(X_m)$ fixes pointwise $\Omega\setminus X_m$ and the action of $R(X_m)$ on $X_m$ is permutation isomorphic to the action of $R_{r-1}$ on $\{0,\ldots,r-2\}$.

Note that for each $i_1,i_2\in \{0,\ldots,s-1\}$ and $j_1,j_2\in \{0,\ldots,m-1\}$, the group $R(Y_{i_1,j_1})$ centralises $R(Y_{i_2,j_2})$ and $R(X_m)$. Set 

$$H'=\left(\prod_{j=0}^{m-1}\prod_{i=0}^{s-1}R(Y_{i,j})\right)\times R(X_m)\cong R^{sm}\times R_{r-1}.$$

For each $\tau\in S$ define the following permutation of $\Sym(\Omega)$
\[
y_{\tau}:\left\{
\begin{array}{lllcl}
z+ir+jrs&\mapsto& z+(i^\tau)r+jrs&&\textrm{for }0\leq z\leq r-1,0\leq i\leq s-1,\\
        &       &            &&0\leq j\leq m-1,\\  
z+mrs&\mapsto&z+mrs&&\textrm{for }0\leq z\leq r-2.   
\end{array}
\right.
\]
Write $$S(\Omega)=\langle x_\tau\mid \tau\in S\rangle.$$ Clearly,  $S(\Omega)$ fixes pointwise $\Omega\setminus X_m$ and, for each $j\in \{0,\ldots,m-1\}$, $S(\Omega)$ fixes setwise $X_j$. Furthermore, the action of $S(\Omega)$ on the partition $\{Y_{0,j},\ldots,Y_{s-1,j}\}$ of $X_j$ is permutation isomorphic to the action of $S$ on $\{0,\ldots,s-1\}$. This yields that $S(\Omega)$ normalises $H'$ and, for each $j\in \{0,\ldots,m-1\}$, the group $$\langle R(Y_{0,j}),\ldots,R(Y_{s-1,j}),S(\Omega)\rangle=(R(Y_{0,j})\times \cdots \times R(Y_{s-1,j}))\rtimes S(\Omega)$$ is isomorphic to $R\wr S$ and its  action on $X_j$ is equivalent to the natural imprimitive action of $R\wr S$ on $\{0,\ldots,rs-1\}$. Write

\begin{eqnarray*}
H&=&\langle H', S(\Omega)\rangle \cong (R^{sm}\rtimes S)\times R_{r-1} .
\end{eqnarray*}

We let $S_0(\Omega)$ denote the subgroup $\langle y_\tau\mid \tau\in S_0\rangle$ of $S(\Omega)$. Clearly, for each $j\in \{0,\ldots,m-1\}$, the group $S_0(\Omega)$ fixes pointwise $Y_{0,j}$ and $S_0(\Omega)\cong S_0$.

Define the following permutation of $\Sym(\Omega)$
\[
a:\left\{
\begin{array}{llll}
z&\mapsto &z+mrs&\textrm{for }0\leq z\leq r-2\\
r-1&\mapsto &r-1&\\
z+jrs&\mapsto&z+(j+1)rs&\textrm{for }0\leq z\leq r-1,\\
     &       &         &1\leq j\leq m-2, j \textrm{ odd}, \\
z+jrs&\mapsto&z+(j-1)rs&\textrm{for }0\leq z\leq r-1,\\
     &       &         &2\leq j\leq m-1, j \textrm{ even}, \\
z+ir+jrs&\mapsto&z+ir+(j+1)rs&\textrm{for }0\leq z\leq r-1,1\leq i\leq s-1,\\
     &       &         &0\leq j\leq m-3, j \textrm{ even}, \\
z+ir+jrs&\mapsto&z+ir+(j-1)rs&\textrm{for }0\leq z\leq r-1,1\leq i\leq s-1,\\
     &       &         &1\leq j\leq m-2, j \textrm{ odd}, \\
z+ir+(m-1)rs&\mapsto&z+ir+(m-1)rs&\textrm{for }0\leq z\leq r-1,1\leq i\leq s-1,\\
z+mrs&\mapsto&z&\textrm{for }0\leq z\leq r-2.
\end{array}
\right.
\]
Write $Y_{0,0}'=Y_{0,0}\setminus\{r-1\}$. Clearly $a$ is an involution of $\Sym(\Omega)$ fixing pointwise $X_{m-1}\setminus Y_{0,m-1}$ and $r-1$, and with
\[
\begin{array}{lllll}
Y_{0,0}'^a=X_m & \textrm{and}&X_m^a=Y_{0,0}',\\
Y_{0,j}^a=Y_{0,j+1}& \textrm{and}&Y_{0,j+1}^a=Y_{0,j}&\textrm{for }1\leq j\leq m-2, j \textrm{ odd},\\
Y_{i,j}^a=Y_{i,j-1}& \textrm{and}&Y_{i,j-1}^a=Y_{i,j}&\textrm{for }1\leq i\leq s-1,1\leq j\leq m-2, j \textrm{ odd}.\\
\end{array}
\]
In particular, $a$ normalises the subgroup $$K'=\left(R(Y_{0,0})_{r-1}\times R(Y_{1,0})\times \cdots \times R(Y_{s-1,0})\right)\times\left(
\prod_{j=1}^{m-1}\prod_{i=0}^{s-1}R(Y_{i,j})
\right) \times R(X_m)$$ of $H'$, centralises the subgroup $S_0(\Omega)$ of $S(\Omega)$, and hence normalises the subgroup $$K=\langle K',S_0(\Omega)\rangle=K'\rtimes S_{0}(\Omega)$$
of $H$.
\begin{lemma}\label{lemma:1}$|H:(H\cap H^a)|=rs$, the core of $H\cap H^a$ in $H$ is 
$$L=\left(\prod_{j=1}^{m-1}\prod_{i=0}^{s-1}R(Y_{i,j})\right)\times R(X_m)$$ and the action of $H/L$ on the right cosets of $(H\cap H^a)/L$ is equivalent to the natural imprimitive action of $R\wr S$ of degree $rs$.
\end{lemma}

\begin{proof}
The orbits of $r-1$ under $H$ and $H^a$ are
$$(r-1)^{H}=\{0,\ldots,rs-1\}=X_0$$
and
\begin{eqnarray*}
(r-1)^{H^a}&=&((r-1)^a)^{Ha}=(r-1)^{Ha}=X_0^a=\left(\bigcup_{i=0}^{s-1}Y_{i,0}\right)^a=\bigcup_{i=0}^{s-1}Y_{i,0}^a\\
&=&(Y_{0,0}'\cup\{r-1\})^a\cup\bigcup_{i=1}^{s-1} Y_{i,1}=X_m\cup \{r-1\}\cup \bigcup_{i=1}^{s-1}Y_{i,1}.
\end{eqnarray*}
Note that 
\begin{eqnarray*}
|H:K|&=&|S(\Omega):S_0(\Omega)||R(Y_{0,0}):R(Y_{0,0})_{r-1}|=|S:S_0||R:R_{r-1}|=rs,
\end{eqnarray*}
that $K$ fixes the point $r-1$ of $\Omega$ and that $(r-1)^{H}$ has size $rs$. This gives that $K$ is 
the stabiliser in $H$ of the point $r-1$ of $\Omega$. 

Let $g\in H\cap H^a$. We have $(r-1)^g\in (r-1)^H\cap (r-1)^{H^a}=\{r-1\}$ and hence $g$ fixes $r-1$. Therefore $H\cap H^a\subseteq K$. Since $a$ normalises $K$, we have $K\subseteq H\cap H^a$ and hence $H\cap H^a=K$ is the stabiliser in $H$ of the point $r-1$ of $\Omega$.

In particular, the action of $H$ on the right cosets of $H\cap H^a$ is permutation isomorphic to the action of $H$ on $(r-1)^H={X_0}$. Therefore the core of $H\cap H^a$ in $H$ is the pointwise stabiliser of $X_0$ in $H$, which is clearly $L$. Finally, as the action of $H$ on $X_0$ is permutation isomorphic to the imprimitive action of $R\wr S$ of degree $rs$, we have that the action of $H/L$ on the right cosets of $(H\cap H^a)/L$ is permutation isomorphic to the imprimitive action of $R\wr S$ of degree $rs$.
\end{proof}

\begin{lemma}\label{lemma:2}$\Alt(\Omega)\subseteq \langle H,a\rangle$.
\end{lemma}
\begin{proof}
Write $V=(\prod_{j=0}^{m-1}\prod_{i=0}^{s-1}R(Y_{i,j}))\rtimes S(\Omega)\subseteq H$. The group $V$ fixes pointwise $X_m$. Furthermore $X_0,\ldots,X_{m-1}$ are the non-trivial orbits of $V$. Now we prove two claims from which the lemma will follow.

\smallskip

\noindent\textsc{Claim~1. }$\langle V,a\rangle$ is transitive on $\Omega$.

\smallskip

\noindent Recall that $X_0,\ldots, X_{m-1}$ are $V$-orbits. If $j\in \{0,\ldots,m-3\}$ is even, then  $a$ maps $r+jrs\in X_j$ to $r+(j+1)rs\in X_{j+1}$, and if $j\in \{1,\ldots,m-2\}$ is odd, then $a$ maps $jrs\in X_j$ to $(j+1)rs\in X_{j+1}$. Therefore $X_0\cup\cdots\cup X_{m-1}$ is contained in the orbit $0^{\langle V,a\rangle}$ of $\langle V,a\rangle$. Finally, as $a$ maps $Y_{0,0}'\subseteq X_0$ to $X_m$, we get that $\langle V,a\rangle$ is transitive on $X_0\cup\cdots\cup X_m=\Omega$.~$\blacksquare$

\smallskip

\noindent\textsc{Claim~2. }$\langle V,a\rangle$ is primitive on $\Omega$.

\smallskip

\noindent We argue by contradiction and we assume that $\langle V,a\rangle$ is imprimitive with a non-trivial system of imprimitivity $\mathcal{B}$. Let $\omega=z+mrs\in X_m$ with $0\leq z\leq r-2$ and $B\in \mathcal{B}$ with $\omega\in B$. Assume that there exists $j\in \{0,\ldots,m-1\}$ with $B\cap X_j\neq \emptyset$.

As $\omega\in B\cap X_m$ and $V$ fixes pointwise $X_m$, we obtain that $B^v=B$ for every $v\in V$. Since $B\cap X_j\neq\emptyset$ and $V$ acts transitively on $X_j$, we obtain that $B$ contains $X_j$. If  $B^a=B$, then $B$ is setwise fixed by $\langle V,a \rangle$, which by Claim~$1$ is transitive, and hence $B=\Omega$, a contradiction. Therefore $B^a\neq B$. Since $a$ fixes pointwise the set $Y_{1,m-1}$ (which is contained in $X_{m-1}$)  and fixes the point $r-1$ (which lies in $X_0$), and since $X_j\subseteq B$, we obtain that $j\neq 0,m-1$. In particular, $1\leq j\leq m-2$. 

Let $\sigma\in R$ with $z^\sigma=r-1$. Note that $ax_{\sigma,0,0}a\in \langle V,a\rangle$. If $j$ is odd, then $$X_j^{ax_{\sigma,0,0}a}=(Y_{0,j+1}\cup (X_{j-1}\setminus Y_{0,j-1}))^{x_{\sigma,0,0}a}=(Y_{0,j+1}\cup (X_{j-1}\setminus Y_{0,j-1}))^{a}=X_j.$$ Similarly, if $j$ is even, then $$X_j^{ax_{\sigma,0,0}a}=(Y_{0,j-1}\cup (X_{j+1}\setminus Y_{0,j+1}))^{x_{\sigma,0,0}a}=(Y_{0,j-1}\cup (X_{j+1}\setminus Y_{0,j+1}))^{a}=X_j.$$
Therefore $ax_{\sigma,0,0}a$ fixes setwise $X_j$. As $X_j\subseteq B$, we have $B^{ax_{\sigma,0,0}a}=B$ and $r-1=\omega^{ax_{\sigma,0,0}a}\in B\cap X_0$, a contradiction.

From the previous contradiction we obtain that $B\cap X_j=\emptyset$ for every $j\in \{0,\ldots,m-1\}$. In particular, every element of $\mathcal{B}$ is either contained in $X_0\cup\cdots \cup X_{m-1}$ or in $X_m$. Let $B_1,\ldots,B_v$ be the blocks of $\mathcal{B}$ contained in $X_m$. We have $X_m=B_1\cup \cdots \cup B_v$, $r-1=v|B_1|$ and $|B_1|$ divides $r-1$. Now $B_1^a\subseteq Y_{0,0}$ is a block of imprimitivity for $R(Y_{0,0})$ on its action on $Y_{0,0}$. Therefore $|B_1|$ divides $|Y_{0,0}|=r$ and hence $|B_1|=1$, a contradiction. Thus  $\langle V,a\rangle$ is primitive.~$_\blacksquare$

\smallskip

Let $G$ be a permutation group on $\Omega$. A subset $\Gamma$ of $\Omega$ is said to be a Jordan set~\cite[Section~$7.4$]{DM} for $G$ if $|\Gamma|>1$ and if there exists a subgroup of $G$ fixing pointwise $\Omega\setminus \Gamma$ and acting transitively on $\Gamma$. 
The $1889$ theorem of B.~Marggraff (see~\cite[Theorem~$7.4$B]{DM}) shows that if $G$ is a primitive permutation group on $\Omega$ with a Jordan set $\Gamma$ with $|\Gamma|<n/2$, then $\Alt(\Omega)\subseteq G$.

The subgroup $R(Y_{0,0})$ of $V$ fixes pointwise $\Omega\setminus Y_{0,0}$ and acts transitively on $Y_{0,0}$. Therefore the set $Y_{0,0}$ is a Jordan set for the primitive group $\langle V,a\rangle$. As $1<|Y_{0,0}|=r<|\Omega|/2$, the theorem of Marggraff yields that $\Alt(\Omega)\subseteq \langle V,a\rangle$.
\end{proof}

\begin{theorem}\label{thm:3}Let $r,s,m,\Omega,H$ and $a$ be as above, $G=\langle H,a\rangle$ and $\Delta=\Cos(G,H,a)$. Then $\Delta$ is a connected $G$-arc-transitive graph of valency $rs$ with $\Alt(\Omega)\subseteq G\subseteq \Sym(\Omega)$. For an arc $(x,y)$ of $\Delta$, we have $G_x^{\Delta(x)}=R\wr S$ (in its imprimitive action of degree $rs$), 
$G_x^{[1]}\cong R^{s(m-1)}\times R_{r-1}$ and 
$G_{x,y}^{[1]}\cong R^{s(m-2)+1}$.
\end{theorem}

\begin{proof}
As $G=\langle H,a\rangle$, $\Delta$ is connected, and since $a$ is an involution, $\Delta$ is undirected. From Lemma~\ref{lemma:2}, $\Alt(\Omega)\subseteq G\subseteq \Sym(\Omega)$. 

Since $\Delta$ is $G$-arc-transitive, we may assume that $(x,y)=(H,Ha)$. From Lemma~\ref{lemma:1}, $\Delta$ has valency $rs$, $G_x^{\Delta(x)}\cong R\wr S$ in its imprimitive action of degree $rs$ and (using the notation in the statement of Lemma~\ref{lemma:1}) $L=G_x^{[1]}\cong R^{s(m-1)}\times R_{r-1}$. Furthermore, $G_{x,y}^{[1]}=G_{x}^{[1]}\cap G_{y}^{[1]}=G_{x}^{[1]}\cap G_{x^a}^{[1]}=G_x^{[1]}\cap (G_x^{[1]})^a=L\cap L^a$. Now, using Lemma~\ref{lemma:1} and the definition of $a$, we have
$$L^a=(R(Y_{0,0})_{r-1}\times R(Y_{1,0})\times\cdots \times R(Y_{s-1,0}))\times R(Y_{0,1})\times \left(\prod_{j=2}^{m-1}\prod_{i=0}^{s-1}R(Y_{i,j})\right)$$
and $$L\cap L^a=R(Y_{0,1})\times \left(\prod_{j=2}^{m-1}\prod_{i=0}^{s-1}R(Y_{i,j})\right)\cong R^{s(m-2)+1},$$
and the theorem is proved.
\end{proof}


\thebibliography{13}

\bibitem{BM}\'{A}. Bereczky and A. Mar\'oti. On groups with every normal
  subgroup transitive or semiregular. \emph{J. Algebra} \textbf{319}
  (2008), no. 4, 1733--1751.  


\bibitem{CPSS}P. J. Cameron, C. E. Praeger, J. Saxl and G. M. Seitz. On
  the Sims conjecture and distance transitive graphs.
  \textit{Bull. Lond. Math. Soc. } \textbf{15} (1983), 499--506.


\bibitem{DM}J.~D.~Dixon and B.~Mortimer. \textit{Permutation Groups} (Springer-Verlag, New York, 1996).

\bibitem{G}A. Gardiner. Arc-transitivity in graphs.
  \emph{Quat. J. Math. Oxford} \textbf{24} (1973), 399--407. 

\bibitem{KS}K. A. Kearnes and \'A. Szendrei. Collapsing permutation groups.
  \textit{Algebra Universalis} \textbf{45} (2001), 35--51.

\bibitem{K}W. Knapp. On the point stabilizer in a primitive
  permutation group. \emph{Math. Z.} \textbf{133} (1973), 137--168.


\bibitem{PSV2}P.~Poto\v{c}nik, P.~Spiga and G.~Verret. Tetravalent arc-transitive graphs with unbounded vertex-stabilisers. \textit{Bull Austr. Math. Soc}, to appear.


\bibitem{PSV}P.~Poto\v{c}nik, P.~Spiga and G.~Verret. On restricted permutation groups, in preparation.

\bibitem{Cheryl}C.~E.~Praeger. Finite quasiprimitive group actions on graphs and designs. In \textit{Groups  Korea ’98}, Eds: Young Gheel Baik, David L. Johnson, and Ann Chi Kim, (de Gruyter, Berlin and New York, 2000), 319--331.

\bibitem{Sabidussi}G. Sabidussi. Vertex-transitive graphs. \textit{Monatsh. Math.} \textbf{68} (1964) 426--438.


\bibitem{Sims}C. C. Sims. Graphs and finite permutation groups.
  \textit{Math. Z. } \textbf{95} (1967), 76--86.

\bibitem{Suzuki}M. Suzuki. \emph{Group Theory II} (Grundlehren
  Math. Wiss. $248$, Springer, $1989$).

\bibitem{Thompson}J. G. Thompson. Bounds on the orders of maximal
  subgroups. \emph{J. Algebra} \textbf{14} (1970), 101--103. 

\bibitem{VanBon}J. Van Bon. Thompson-Wielandt-like theorems
  revisited. \emph{Bull. London Math. Soc.} \textbf{35} (2003), 30--36.

\bibitem{Weiss}R.~Weiss. $s$-transitive graphs. \emph{Colloq. Math. Soc. J\'anos Bolyai} \textbf{25} (1978), 827--847.


\bibitem{W}R. Weiss. Elations of graphs. \emph{Acta Math. Hungar.}
  \textbf{34} (1979), 101--103.



\bibitem{Wielandt}H. Wielandt. \textit{Subnormal subgroups and permutation
  groups} (Lecture Notes, Ohio State University, Columbus, Ohio, 1971).
\end{document}